\documentclass{amsart}
\usepackage[utf8]{inputenc}
\usepackage{amsmath, amssymb}
\usepackage{amsthm}
\usepackage[all]{xy}
\usepackage{color}
\definecolor{darkblue}{rgb}{0,0,0.6}
\usepackage[ocgcolorlinks,colorlinks=true, citecolor=darkblue, filecolor=darkblue,
linkcolor=darkblue, urlcolor=darkblue]{hyperref}
\usepackage[nameinlink,capitalise,noabbrev]{cleveref}
\usepackage{bbm}

\makeatletter
\@namedef{subjclassname@2010}{%
  \textup{2010} Mathematics Subject Classification}
\makeatother
\title{On the K-theory of linear groups}
\author{Daniel Kasprowski}
\address{Max-Planck-Institut für Mathematik, Vivatsgasse 7, 53111 Bonn, Germany}
\email{kasprowski@mpim-bonn.mpg.de}
\date{\today}

\newcommand{\bbC}{\mathbbm{C}}
\newcommand{\bbQ}{\mathbbm{Q}}

\newcommand{\bbF}{\mathbb{F}}

\newcommand{\bbZ}{\mathbbm{Z}}
\newcommand{\bbN}{\mathbbm{N}}
\newcommand{\bbL}{\mathbb{L}}
\newcommand{\bbK}{\mathbb{K}}

\newcommand{\mcX}{\mathcal{X}}
\newcommand{\mcY}{\mathcal{Y}}
\newcommand{\mcZ}{\mathcal{Z}}
\newcommand{\mcn}{\mathfrak{n}}
\newcommand{\mcp}{\mathfrak{p}}
\newcommand{\mfD}{\mathfrak{D}}
\newcommand{\mfB}{\mathfrak{B}}

\DeclareMathOperator{\diam}{diam}

\DeclareMathOperator{\Map}{Map}

\DeclareMathOperator{\asdim}{asdim}
\DeclareMathOperator*{\colim}{colim}

\newcommand{\mcD}{\mathcal{D}}
\newcommand{\mcF}{\mathcal{F}}
\newcommand{\mcS}{\mathcal{S}}
\newcommand{\mcA}{\mathcal{A}}

\newcommand{\Fin}{{\mathcal{F}in}}

\newcommand{\comment}[1]{}

\numberwithin{equation}{section}
\newtheorem{thm}[equation]{Theorem}

\newtheorem{prop}[equation]{Proposition}

\newtheorem{lemma}[equation]{Lemma}

\newtheorem{cor+}{Corollary}
\newtheorem{prop+}[cor+]{Proposition}
\theoremstyle{definition}

\newtheorem{defi}[equation]{Definition}

\newtheorem{rem}[equation]{Remark}

\newtheorem{notation}[equation]{Notation}

 \newtheoremstyle{TheoremNum}
        {}{}              %%% space between body and thm
        {\itshape}                      %%% Thm body font
        {}                              %%% Indent amount (empty = no indent)
        {\bfseries}                     %%% Thm head font
        {.}                             %%% Punctuation after thm head
        { }                             %%% Space after thm head
        {\normalfont\bfseries\thmname{#1}\thmnote{#3}}%%% Thm head spec
\theoremstyle{TheoremNum}

\begin{document}
\begin{abstract}
We prove that for a finitely generated linear group over a field of positive characteristic the family of quotients by finite subgroups has finite asymptotic dimension. We use this to show that the $K$-theoretic assembly map for the family of finite subgroups is split injective for every finitely generated linear group $G$ over a commutative ring with unit under the assumption that $G$ admits a finite-dimensional model for the classifying space for the family of finite subgroups. Furthermore, we prove that this is the case if and only if an upper bound on the rank of the solvable subgroups of $G$ exists.
\end{abstract}
\maketitle
\section{Introduction}
For every group $G$ and every ring $A$ there is a functor $\bbK_A\colon OrG\to \mathfrak{Spectra}$ from the orbit category of $G$ to the category of spectra sending $G/H$ to (a spectrum weakly equivalent to) the $K$-theory spectrum $\bbK(A[H])$ for every subgroup $H\leq G$. For any such functor $F\colon OrG\to \mathfrak{Spectra}$, a $G$-homology theory $\bbF$ can be constructed via
\[\bbF(X):=\Map_G(\_,X_+)\wedge_{OrG}F,\]
see Davis and Lück \cite{davislueck}. We will denote its homotopy groups by $H_n^G(X;F):=\pi_n\bbF(X)$. The assembly map for the family of finite subgroups is the map
\begin{equation*}
H_n^G(\underbar EG;\bbK_A)\to H^G_n(pt;\bbK_A)\cong K_n(A[G])\label{eq1}
\end{equation*}
induced by the map $\underbar EG\to pt$. Here $\underbar EG$ denotes the classifying space for the family of finite subgroups, see Lück \cite{MR1757730}. The assembly map is a helpful tool to relate the $K$-theory of the group ring $A[G]$ to the $K$-theory of the group rings over the finite subgroups $H\leq G$. It can more generally be defined for any additive $G$-category instead of $A$, see Bartels and Reich \cite{coefficients}. Note that additive categories will always be small and that K-theory will always mean non-connective K-theory constructed by Pedersen and Weibel \cite{MR802790}.
\begin{thm}
\label{thm:main}
Let $R$ be a commutative ring with unit and let $G\leq GL_n(R)$ be finitely generated. If $G$ admits a finite-dimensional model for the classifying space $\underbar EG$, then the assembly map
\[H_n^G(\underbar EG;\bbK_\mcA)\to K_n(\mcA[G])\]
is split injective for every additive $G$-category $\mcA$.

If $\mcA$ is an additive $G$-category with involution such that for every virtually nilpotent subgroup $A\leq G$ there exists $i_0\in\bbN$ such that for $i\geq i_0$ we have $K_{-i}(\mcA[A])=0$, then the $L$-theoretic assembly map
\[H_n^G(\underbar EG;\bbL^{\langle-\infty\rangle}_\mcA)\to L^{\langle-\infty\rangle}_n(\mcA[G])\]
is split injective.
\end{thm}
\cref{thm:main} implies the (generalized integral) Novikov conjecture for these groups by \cite[Section 6]{KasLie}, since virtually nilpotent groups satisfy the Farrell--Jones conjecture by Wegner \cite{solvable}. The (rational) Novikov conjecture for these groups is already known by Guentner, Higson and Weinberger \cite{GHWNovikovlinear}, where it is shown that the Baum-Connes assembly map is split injective for linear groups.

We will use inheritance properties to reduce the proof of the theorem to the case where the ring $R$ has trivial nilradical and show that in this case the family $\{F\backslash G\}_{F\in\Fin}$ has finite decomposition complexity, where $\Fin$ denotes the family of finite subgroups of $G$. Then the theorem follows from the main theorem of \cite{thesis}. For convenience, the necessary results of \cite{thesis} are recalled in the appendix.

By a result of Alperin and Shalen \cite{alperin} a finitely generated subgroup $G$ of $GL_n(F)$, where $F$ is a field of characteristic zero, has finite virtual cohomological dimension if and only if there is a bound on the Hirsch rank of the unipotent subgroups of $G$. This in particular implies that it has a finite-dimensional model for the classifying space $\underbar EG$. In positive characteristic, a finitely generated subgroup $G\leq GL_n(F)$ always admits a finite-dimensional model for $\underbar EG$ by Degrijse and Petrosyan \cite[Corollary 5]{poschar}. We generalize this in the following way.
\begin{prop}
\label{prop:dim}
Let $R$ be a commutative ring with unit and let $G\leq GL_n(R)$ be finitely generated. Then $G$ admits a finite-dimensional model for $\underbar EG$ if and only if there exists $N\in\bbN$ such that $l(A)\leq N$ for every solvable subgroup $A\leq G$, where $l(A)$ denotes the Hirsch rank of $A$.
\end{prop}

Let $G$ be a solvable group and $1=G_0\trianglelefteq G_1\trianglelefteq\ldots G_{n-1}\trianglelefteq G_n=G$ a normal series with abelian factors. The \emph{Hirsch rank} (or \emph{Hirsch length}) $l(G)$ of $G$ is 
\[l(G)=\sum_{i=1}^n\dim_\bbQ\bbQ\otimes_\bbZ(A_i/A_{i-1}).\]

We will prove the proposition in \cref{sec:dimension}.\\

\textbf{Acknowledgments:} I would like to thank Christoph Winges and the referee for useful comments and suggestions. This work was supported by the Max Planck Society.
\section{Finite decomposition complexity}
Let $X$ be a metric space. A decomposition $X=\bigcup_{i\in I} U_i$ is called \emph{$r$-disjoint}, if $d(U_i,U_j)>r$ for all $i\neq j\in I$. We then denote the decomposition by \[X=\bigcup^{r\text{-disj.}}U_i.\]
A \emph{metric family} is a set of metric spaces. A metric family $\{X_i\}_{i\in I}$ has \emph{finite asymptotic dimension uniformly} if there exists an $n\in\bbN$ such that for every $r>0,$ $i\in I$ there exists decompositions $X_i=\bigcup_{k=0}^n U_{i}^k$ and
\[U_i^k=\bigcup_{j\in J_{i,k}}^{r\text{-disj.}}U_{i,j}^k\]
such that $\sup_{i,j,k}U_{i,j}^k<\infty$.

In \cite{FDC} Guentner, Tessera and Yu introduced the following generalization of finite asymptotic dimension.
\begin{defi}
\label{def:fdc}
Let $r>0$. A metric family $\mcX=\{X_\alpha\}_{\alpha\in A}$ \emph{$r$-decomposes over a class of metric families $\mfD$} if for every $\alpha\in A$ there exists a decomposition $X_\alpha=U^r_\alpha\cup V^r_\alpha$ and $r$-disjoint decompositions
\[U_\alpha^r=\bigcup^{r\text{-disj.}}_{i\in I(r,\alpha)}U_{\alpha,i}^r,\qquad V_\alpha^r=\bigcup^{r\text{-disj.}}_{j\in J(r,\alpha)}V_{\alpha,j}^r\]
such that the families $\{U_{\alpha,i}^r\}_{\alpha\in A,i\in I(r,\alpha)}$ and $\{V_{\alpha,j}^r\}_{\alpha\in A,j\in J(r,\alpha)}$ lie in $\mfD$. A metric family $\mcX$ \emph{decomposes over $\mfD$} if it $r$-decomposes over $\mfD$ for all $r>0$.

Let $\mfB$ denote the class of \emph{bounded families}, i.e. $\mcX\in\mfB$ if there exists $R>0$ such that $\diam X<R$ for all $X\in\mcX$. We set $\mfD_0=\mfB$. For a successor ordinal $\gamma+1$ we define $\mfD_{\gamma+1}$ to be the class of all metric families which decompose over $\mfD_\gamma$. For a limit ordinal $\lambda$ we define
\[\mfD_\lambda=\bigcup_{\gamma<\lambda}\mfD_\gamma.\]
A metric family $\mcX$ has \emph{finite decomposition complexity (FDC)} if $\mcX\in \mfD_\gamma$ for some ordinal~$\gamma$.

A metric space $X$ has FDC if the family $\{X\}$ consisting only of $X$ has FDC. A group $G$ has FDC if it has FDC with any (and thus every) proper left-invariant metric.
\end{defi}
A subfamily $\mcZ$ of a metric family $\mcY$ is a metric family $\mcZ$ such that for each $Z\in \mcZ$ there exists an $Y\in\mcZ$ with $Y\subseteq X$.

A map $F\colon \mcX\to \mcY$ is between metric families $\mcX,\mcY$ is a set of maps from elements of $\mcX$ to elements of $\mcY$ such that every $X\in\mcX$ is the domain of at least one $f\in F$. The inverse image $F^{-1}(\mcZ)$ of a subfamily $\mcZ$ of $\mcY$ is the metric family $\{f^{-1}(Z)\mid Z\in\mcZ, f\in F\}$.
A map $F\colon \mcX\to \mcY$ is called \emph{uniformly expansive}, if there exists a non-decreasing function $\rho\colon[0,\infty)\to[0,\infty)$ such that for every $f\colon X\to Y$ in $F$ and every $x,y\in X$ we have
\[d(f(x),f(y))\leq \rho(d(x,y)).\]
We will use the following three results about FDC.
\begin{thm}[{\cite[Fibering Theorem 3.1.4]{FDC}}]
Let $\mcX$ and $\mcY$ be metric families and let $F\colon \mcX\to \mcY$ be uniformly expansive. Assume $\mcY$ has FDC and that for every bounded subfamily $\mcZ$ of $\mcY$ the inverse image $F^{-1}(\mcZ)$ has FDC. Then $\mcX$ also has FDC.
\end{thm}
\begin{thm}[{\cite[Theorem 4.1]{FDC}}]
A metric space $X$ with finite asymptotic dimension has FDC.
\end{thm}
While the above theorem is stated only for metric spaces it also holds for metric families which have finite asymptotic dimension uniformly.
\begin{thm}[{\cite[3.1.7]{FDC}}]
\label{thm:union}
Let $X$ be a metric space, expressed as a union of finitely many subspaces $X=\bigcup_{i=0}^nX_i$. If the metric family $\{X_i\}_{i=0,\ldots,n}$ has FDC, so does $X$.
\end{thm}
This theorem again holds for metric families instead of metric spaces, i.e. a metric family $\{\bigcup_{i=0}^nX_{ij}\}_{j\in J}$ has FDC if and only if the family $\{X_{ij}\}_{j\in J,i=0,\ldots,n}$ has FDC. %or equivalently the families $\{X_{0j}\}_{j\in J},\ldots\{X_{nj}\}_{j\in J}$ have FDC.
We will also need the following two results about finite asymptotic dimension.
\begin{lemma}
\label{thm:roe}
Let $P\colon \mcX\to \mcY$ be a family of maps such that there is $k>0$ and each $p\in P$ is $k$-Lipschitz. Suppose that $\mcY$ has finite asymptotic dimension uniformly and that for each $R>0$ the family \[\{p^{-1}(B_R(y))\mid X\in \mcX, Y\in \mcY, y\in Y, (p\colon X\to Y)\in P\}\] has finite asymptotic dimension uniformly. Then $\mcX$ has finite asymptotic dimension uniformly.
\end{lemma}
This is \cite[Lemma 9.16]{roe} for metric families instead of metric spaces. The proof is the same.
\begin{lemma}
\label{thm:roeunion}
Let $\mcX=\{U_\alpha\cup V_\alpha\}_{\alpha\in A}$ be a metric family. Then
\[\asdim \mcX=\max\{\asdim\{U_\alpha\}_{\alpha},\asdim\{V_\alpha\}_{\alpha\in A}\}.\]
\end{lemma}
This is \cite[Proposition 9.13]{roe} for metric families instead of metric spaces. The proof is the same.

In the next section it will be more convenient to work with pseudometrics instead of metrics, i.e. allowing $d(x,y)=0$ for $x\neq y$. Finite asymptotic dimension and FDC are defined in the same way for pseudometrics. If $d$ is a pseudometric on $X$, then we can define a metric $d'$ on $X$ by setting $d'(x,y):=\max\{1,d(x,y)\}$ for $x\neq y$. The metric $d'$ is proper resp. left-invariant if and only if $d$ is. It has finite asymptotic dimension resp. FDC if and only if $d$ does. Therefore, to show that a group has finite asymptotic dimension or FDC, it suffices to show this for $G$ equipped with a left-invariant proper pseudometric.
\begin{notation}
If $G$ is a group, then by $\{F\backslash G\}_{F\in\Fin}$ we will always mean the family of quotients by all finite subgroups of $G$, i.e. $\Fin$ will always refer to the family of finite subgroups of the group of which we take the quotients.
\end{notation}
\section{Linear groups over fields of positive characteristic}
In this section $K$ will always denote a field of positive characteristic. Every finitely generated subgroup $G$ of $GL_n(K)$ has finite asymptotic dimension by \cite[Theorem 3.1]{rigidity}. Here we want to show that even the family $\{F\backslash G\}_{F\in\Fin}$ has finite asymptotic dimension uniformly. We begin by recalling the argument from \cite{rigidity}. 

A \emph{length function} on a group $G$ is function $l\colon G\to[0,\infty)$ such that for all $g,h\in G$
\begin{enumerate}
\item $l(e)=0$,
\item $l(g)=l(g^{-1})$, and
\item $l(gh)\leq l(g)+l(h)$.
\end{enumerate}
We do not require that $l$ is proper, nor that $l(g)=0$ if and only if $g=e$. By setting $d(g,h):=l(g^{-1}h)$ we obtain a pseudometric.

A \emph{discrete norm} on a field $K$ is a map $\gamma\colon K\to[0,\infty)$ satisfying that for all $x,y\in K$ we have
\begin{enumerate}
\item $\gamma(x)=0$ if and only if $x=0$,
\item $\gamma(xy)=\gamma(x)\gamma(y)$,
\item $\gamma(x+y)\leq \max\{\gamma(x),\gamma(y)\}$
\end{enumerate}
and that the range of $\gamma$ on $K\setminus\{0\}$ is a discrete subgroup of the multiplicative group $(0,\infty)$.

Following \cite{GHWNovikovlinear}, we obtain for every discrete norm $\gamma$ on $K$ a length function $l_\gamma$ on $GL_n(K)$ by
\[l_\gamma(g)=\log\max_{i,j}\{\gamma(g_{ij}),\gamma(g^{ij})\},\]
where $g_{ij}$ and $g^{ij}$ are the matrix coefficients of $g$ and $g^{-1}$, respectively.
By \cite[Propostion 5.2.4]{FDC} the group $GL_n(K)$ equipped with the pseudometric $d(g,h)=l_\gamma(g^{-1}h)$ has finite asymptotic dimension for every discrete norm $\gamma$. Let us review the proof.

The subset $\mathcal O:=\{x\in K\mid \gamma(x)\leq 1\}$ is a subring of $K$ called the \emph{ring of integers} and $\mathfrak m:=\{x\in K\mid \gamma(x)<1\}$ is a principal ideal in $\mathcal O$. Let $\pi$ be a fixed generator of $\mathfrak m$ and let $D$ denote the subgroup of diagonal matrices with powers of $\pi$ on the diagonal. Let $U$ denote the unipotent upper triangular matrices. By \cite[Lemma 5.2.5]{FDC} the group $U$ has asymptotic dimension zero.
We have $D\cong\bbZ^n$ and the restriction of $l_\gamma$ to $D$ is given by
\[l_\gamma(a):=\max_i |k_i|\log \gamma(\pi^{-1})\]
where $a$ is the diagonal matrix with entries $\pi^{k_i}$ on the diagonal. The group $D$ therefore is quasi-isometric to $\bbZ^n$ with the standard metric and has asymptotic dimension $n$. Setting $T:=DU$, $T$ is again a subgroup of $GL_n(K)$ and $U\leq T$ is normal. Considering the extension $1\to U\to T\to D\to 1$, we see that $T$ has finite asymptotic dimension.

Let $H$ be the subgroup of those $g\in GL_n(F)$ for which the entries of $g$ and $g^{-1}$ are in $\mathcal O$. Then $GL_n(K)=TH$ by \cite[Lemma 5]{GHWNovikovlinear}. For $h\in H$ let $h_{ij}$ and $h^{ij}$ denote the matrix coefficients of $h$ and $h^{-1}$ respectively. By definition $\gamma(h_{ij}),\gamma(h^{ij})\leq 1$ and thus \[0\leq l_\gamma(h)=\log\max_{ij}\{\gamma(h_{ij}),\gamma(h^{ij})\}\leq 0.\]
This implies that the inclusion $T\to GL_n(K)$ is isometric and metrically onto, i.e. for every $g\in GL_n(K)$ there exists a $t\in T$ with $d(g,t)=0$. Hence, $GL_n(K)$ has finite asymptotic dimension with respect to the pseudometric $d$.

\begin{lemma}
\label{lem:unipotent}
For every discrete norm the family $\{F\backslash GL_n(K)\}_F$, where $F$ ranges over all finite subgroups of $U$, has finite asymptotic dimension uniformly with respect to the associated pseudometric.
\end{lemma}
\begin{proof}
Let $F$ be a finite subgroup of $U$. Then we can consider the map
\[F\backslash T\xrightarrow{\rho_F}D.\]
We want to apply \cref{thm:roe} to the family $\{\rho_F\colon F\backslash T\xrightarrow{\rho}D\}_{F\leq U~fin.}$. For this we have to show that for every $R>0$ the family $\{\rho_F^{-1}(B_R(d))\}_{d\in D,F\leq U~fin.}$ has finite asymptotic dimension uniformly. 
The preimage $\rho_F^{-1}(d)=\{Fud\mid u\in U\}$ of a point $d\in D$ is isometric to $(F)^d\backslash U$ by mapping $Fud$ to $d^{-1}Fdd^{-1}ud$, where $(F)^d:=\{d^{-1}fd\mid f\in F\}$. Therefore, the preimage of $B_R(d)$ for any $R>0$ is a finite union of spaces isometric to spaces of the form $(F)^{d'}\backslash U$ with $d'\in D$. The number of spaces appearing in this union only depends on $R$ and not on $d$ (or $F$). Thus by \cref{thm:roeunion}
\[\asdim\{\rho_F^{-1}(B_R(d))\}_{d\in D,F\leq U~fin.}=\asdim\{F\backslash U\}_{F\leq U~fin.}.\]

Since the inclusion $F\backslash T\to F\backslash GL_n(K)$ is isometric and metrically onto, to prove the lemma it remains to show that the family $\{F\backslash U\}_{F\leq U~fin.}$ has asymptotic dimension zero uniformly.

Let $R>0$ be given and let $\mcS$ denote the partition of $U$ into $R$-connected components, i.e. two elements $u,u'\in U$ lie inside the same $S\in\mcS$ if and only if there exists a sequence $u_0,\ldots,u_n$ with $u=u_0$, $u'=u_n$ and $d(u_{i-1},u_{i})\leq R$ for all $i=1,\ldots,n$. Since $U$ has asymptotic dimension zero we have that $r:=\sup_{S\in\mcS}\diam S<\infty$. Since the left action of $F$ on $U$ is isometric, if $fu=u'$ for some $f\in F,u,u'\in U$, then $f$ maps the $R$-connected component of $u$ bijectively onto the $R$-connected component of $u'$. This implies that every $R$-connected component of $F\backslash U$ is a quotient of an $R$-connected component of $U$ and in particular has diameter at most $r$. Therefore, the family $\{F\backslash U\}_{F\leq U~fin.}$ has asymptotic dimension zero uniformly as claimed.
\end{proof}
\begin{prop}
\label{prop:norms}
Let $G\leq GL_n(K)$ be a finitely generated subgroup. Then for every discrete norm $\gamma$ the family $\{F\backslash G\}_{F\in\Fin}$ has finite asymptotic dimension uniformly with respect to the associated pseudometric.
\end{prop}
\begin{proof}
By the main theorem of Alperin \cite{alperinselberg} there exists a normal subgroup $G'\leq G$ with index $[G:G']=:N<\infty$ such that every finite subgroup of $G'$ is unipotent. Therefore, every finite subgroup $F\leq G'$ is conjugate in $GL_n(K)$ to a finite subgroup $F'\leq U$. Let $g=th$ with $t\in T,h\in H$ be such that $g^{-1}F'g=F$. Since $U$ is normal in $T$, we have that $t^{-1}F't\leq U$ and we can assume $g\in H$ and in particular $l_\gamma(g)=0$. This implies that conjugation by $g$ is an isometry and induces an isometry between $F'\backslash GL_n(K)$ and $F\backslash GL_n(K)$. By \cref{lem:unipotent} the family $\{F'\backslash GL_n(K)\}_{F'\leq U~fin.}$ has finite asymptotic dimension uniformly and by the above isometry therefore the family $\{F\backslash GL_n(K)\}_{F\leq G'~fin.}$ also has finite asymptotic dimension uniformly. This also holds for the subfamily $\{F\backslash G\}_{F\leq G'~fin.}$. Since $[G:G']=N$ every finite subgroup $\tilde F$ of $G$ has a normal subgroup $F$ of index at most $N$ lying in $G'$. The quotient group $F\backslash \tilde F$ acts isometrically on $F\backslash G$. Thus, projecting the covers that give finite asymptotic dimension for $\{F\backslash G\}_{F\leq G'~fin.}$ down to the quotient $\{\tilde F\backslash G\}_{\tilde F\in\Fin}$ shows that this family still has finite asymptotic dimension uniformly.
\end{proof}
\begin{thm}
\label{prop:linear}
Let $G\leq GL_n(K)$ be a finitely generated subgroup. There exists a proper, left-invariant metric on $G$ such that the family $\{F\backslash G\}_{F\in\Fin}$ has finite asymptotic dimension uniformly.
\end{thm}
\begin{proof}
The subring of $K$ generated by the matrix entries of a finite generating set for $G$ is a finitely generated domain $A$ with $G\leq GL_n(A)$ and we may replace $K$ by the (finitely generated) fraction field of $A$, thus we can assume that $K$ is a finitely generated field of positive characteristic. By \cite[Proposition 3.4]{rigidity} for every finitely generated subring $A$ of $K$ there exists a finite set $N_A$ of discrete norms such that for every $k\in\bbN$ the set
\[B_A(k)=\{a\in A\mid \forall \gamma\in N_A:\gamma(a)\leq e^k\}\]
is finite. Let $A$ again be the subring generated by the matrix entries of a finite generating set for $G$ and $N_A=\{\gamma_1,\ldots,\gamma_q\}$ the finite set of discrete norms as above. Consider the length function $l:=l_{\gamma_1}+\ldots+l_{\gamma_q}$. The pseudometric on $G$ defined by $d(g,g'):=l(g^{-1}g')$ now is proper and left-invariant, and the diagonal embedding
\[(G,d)\to (GL_n(K),d_{\gamma_1})\times\ldots\times (GL_n(K),d_{\gamma_q})\]
is isometric when the product is given the sum metric.
It suffices to show that the family
\[\{F\backslash ((G,d_{\gamma_1})\times\ldots\times (G,d_{\gamma_q}))\}_{F\leq G~fin.}\]
has finite asymptotic dimension uniformly.
Now let $F\leq G$ be finite and consider the projection
\[F\backslash(G\times\ldots\times G)\xrightarrow{p} F\backslash G\times\ldots\times F\backslash G\]
using the same metrics as above. The image has finite asymptotic dimension uniformly in $F$ by \cref{prop:norms} and using \cref{thm:roe} it suffices to show that the preimage of $B_R(Fg_1)\times\ldots\times B_R(Fg_n)$ under $p$ has finite asymptotic dimension uniformly. The preimage is a finite union of metric spaces of the form $F\backslash (Fg'_1\times Fg'_n)$ and the number of the spaces appearing in the union only depends on $R$ not on $F$ or $g_1,\ldots,g_n$.
By the main theorem of \cite{alperinselberg} there exists a normal subgroup $G'\unlhd GL_n(A)$ with index $[GL_n(A):G']=:N<\infty$ such that every finite subgroup of $G'$ is unipotent. In particular, we have a normal unipotent subgroup $F':=G'\cap F$ of $F$ of index at most $N$. The space $Fg'_1\times\ldots\times Fg'_n$ is a union of at most $N$ subspaces isometric to $F'g'_1\times\ldots\times F'g_n'$ and as in the proof of \cref{prop:norms} there exists an isometry of these to $F_1\times\ldots\times F_n$ with $F_i\leq U$. By \cref{thm:roeunion} this shows that $Fg'_1\times\ldots\times Fg'_n$ has asymptotic dimension zero uniformly in $F$. As in the proof of \cref{lem:unipotent}, we see that $F\backslash(Fg'_1\times\ldots\times Fg'_n)$ also has asymptotic dimension zero. This completes the proof of \cref{prop:linear}.
\end{proof}
\begin{rem}
Note that the family $\{F\backslash G\}_{F\in\mcF in}$ has finite asymptotic dimension uniformly for some proper, left-invariant (pseudo)metric on $G$ if and only if it has finite asymptotic dimension for every proper, left-invariant metric on $G$.
\end{rem}
\section{Linear groups over commutative rings with unit}
\begin{lemma}
\label{lem:prod}
Let $H_1,H_2$ be groups such that $\{F\backslash H_i\}_{F\in\Fin}$ has FDC for $i=1,2$. Then $\{F\backslash(H_1\times H_2)\}_{F\in\Fin}$ has FDC.
\end{lemma}
\begin{proof}
Let proper, left-invariant metrics $d_i$ on $H_i$ be given and consider $H_1\times H_2$ with the metric $d_1+d_2$.
Let $p_i\colon H_1\times H_2\to H_i$ denote the projection. Consider the uniformly expansive map 
\[\{F\backslash(H_1\times H_2)\}_{F\in\Fin}\to \{(p_1(F)\times p_2(F))\backslash (H_1\times H_2)\}_{F\in\Fin}.\]
Then the range has FDC by assumption and by the fibering theorem \cite[Theorem 3.1.4]{FDC} it suffices to show that the family 
\[\{F\backslash (p_1(F)\times p_2(F))(B_R(h_1)\times B_R(h_2))\}_{h_i\in H_i,F\leq H_1\times H_2~fin.}\]
has FDC for every $R>0$.
Every space in this family is a union of $|B_R(h_1)\times B_R(h_2)|$ many spaces of the form $F\backslash (p_1(F)\times p_2(F))(h,h')$. The number $|B_R(h_1)\times B_R(h_2)|$ only depends on $R$ not on $h_1$ and $h_2$, and every space $F\backslash (p_1(F)\times p_2(F))(h,h')$ is isometric to $(F)^{(h,h')}\backslash (p_1(F)\times p_2(F))^{(h,h')}$, where $(F)^{(h,h')}$ denotes $(h,h')^{-1}F(h,h')$ and similarly for $(p_1(F)\times p_2(F))^{(h,h')}$. By \cref{thm:union} it suffices to show that the family $\{F\backslash F'\}_{F\leq F'}$ has FDC where $F\leq F'$ ranges over all pairs of finite subgroups of $H_1\times H_2$.
Let $\mcS_R$ denote the family of finite subgroups of $H_1\times H_2$ generated by elements from $B_R(e)$ and let $s_R:=\sup_{S\in\mcS_R}\diam S$. Let $F\leq H_1\times H_2$ be finite. Then for every $R>0$ the group $F$ is the $R$-disjoint union of the cosets of $\langle F\cap B_R(e)\rangle$ and each of these has diameter at most $s_R$. We see that the family of finite subgroups of $H_1\times H_2$ has asymptotic dimension zero uniformly. This implies that the above family $\{F\backslash F'\}_{F\leq F'}$ also has asymptotic dimension zero uniformly since every $R$-connected component of $F\backslash F'$ is a quotient of an $R$-connected component of $F'$ and thus has uniformly bounded diameter. 
\end{proof}
\begin{lemma}[{\cite[Lemma 5.2.3]{FDC}}]
\label{lem:domains}
Let $R$ be a finitely generated commutative ring with unit and let $\mcn$ be the nilpotent radical of $R$,
\[\mcn=\{r\in R\mid \exists n:r^n=0\}.\]
The quotient ring $S=R/\mcn$ contains a finite number of prime ideals $\mcp_1,\ldots,\mcp_k$ such that the diagonal map 
\[S\rightarrow S/\mcp_1\oplus\ldots\oplus S/\mcp_k\]
embeds $S$ into a finite direct sum of domains.
\end{lemma}
\begin{thm}
\label{thm:fdc}
Let $R$ be a commutative ring with unit and trivial nilradical and let $G$ be a finitely generated subgroup of $GL(n,R)$. Then $\{F\backslash G\}_{F\in\Fin}$ has FDC.
\end{thm}
\begin{proof}
Because $G$ is finitely generated we can assume that $R$ is finitely generated as well. Since the nilradical of $R$ is trivial, we have $R=S$ in the notation of the previous lemma and there is an embedding
\[GL_n(S)\hookrightarrow GL_n(S/\mcp_1)\times\ldots\times GL_n(S/\mcp_k)\hookrightarrow GL_n(Q(S/\mcp_1))\times\ldots\times GL_n(Q(S/\mcp_k))\]
where $Q(S/\mcp_i)$ is the quotient field of $S/\mcp_i$. Let $G_i$ be the image of the group $G$ in $GL_n(Q(S/\mcp_i))$. If $S/\mcp_i$ has positive characteristic, the family $\{F\backslash G_i\}_{F\in\Fin}$ has FDC by \cref{prop:linear}. If $S/\mcp_i$ has characteristic zero, then $G_i$ is virtually torsion-free by Selberg's Lemma and thus $\{F\backslash G_i\}_{F\in\Fin}$ has FDC by \cite[Theorem 4.10]{KasFDC}.
Now \cref{lem:prod} implies that the family $\{F\backslash G\}_{F\in\Fin}$ also has FDC.
\end{proof}
\begin{proof}[Proof of \cref{thm:main}:]
This follows directly from \cref{thm:fdc} and \cite[Theorems 3.2.2 and 3.3.1]{thesis} if $R$ has trivial nilradical. Note that \cite[Theorems 3.2.2 and 3.3.1]{thesis} are stronger than the similar \cite[Theorem A and Theorem 9.1]{KasFDC}, where an upper bound on the order of the finite subgroups is needed. For convenience we show in the appendix how the results from \cite{KasFDC} can be used to prove the theorems from \cite{thesis}.

If the nilradical $\mcn$ of $R$ is non-trivial, we have an exact sequence
\[1\to(1+M_n(\mcn))\cap G\to G\to H\to 1\]
where $H$ denotes the image of $G$ in $GL_n(R/\mcn)$. Now the $K$-theoretic assembly map for $H$ is split injective and $(1+M_n(\mcn))\cap G$ is nilpotent. Therefore, the preimage of every virtually cyclic subgroup of $H$ is virtually solvable and satisfies the Farrell--Jones conjecture by \cite{solvable}. By \cite[Proposition 4.1]{KasLie} this implies that the $K$-theoretic assembly map for $G$ is split injective as well. The $L$-theory version of the theorem follows in the same way from the results in \cite[Section 6]{KasLie}.
\end{proof}
\section{Dimension of the classifying space}
\comment{Statt Beweis von Satz vlt Satz 5.16 aus Survey zitieren.}
\label{sec:dimension}
In this section we want to prove \cref{prop:dim}. 
We will need the following result about classifying spaces. The proof is the same as the proof of Lück \cite[Theorem 3.1]{MR1757730}.
\begin{thm}
\label{thm:dim}
Let $1\to K\to G\stackrel{\pi}{\to} Q\to 1$ be an exact sequence of groups. Assume that $Q$ admits a finite-dimensional model for $\underbar EQ$ and that there exists an $N\in\bbN$ such that for every finite subgroup $F\in Q$ the preimage admits a model for $\underbar E\pi^{-1}(F)$ of dimension at most $N$. Then there exists a finite-dimensional model for $\underbar EG$.
\end{thm}
\begin{proof}[Proof of \cref{prop:dim}]
For a group $G$ let $\underline{cd}G$ be the shortest length of a projective resolution of $\bbZ$ as a $\bbZ[G]$-module and let $\underline{hd}G$ be the shortest length of a flat resolution of $\bbZ$ of $\bbZ$ as a $\bbZ[G]$-module. Let $\underline{gd}G$ denote the minimal dimension of a model for $\underbar EG$. For a countable group $G$ by Nucinkis \cite[Theorem 4.1]{MR2061566} we have
\[\underline{hd}G\leq\underline{cd}G\leq\underline{hd}G+1.\]
Furthermore,
\[\underline{cd}G\leq \underline{gd}G\leq\max\{\underline{cd}G,3\},\]
where the first inequality follows from taking the cellular chain complex of $\underbar EG$ as a resolution and the second inequality follows from Lück \cite[Theorem 13.19]{transformation}.
By Flores and Nucinkis \cite[Theorem 1]{MR2280168} for a solvable group with finite Hirsch length $l(G)$ it holds that $l(G)=\underline{hd}G$. Note that Flores and Nucinkis use Hillman's definition of the Hirsch rank for elementary amenably group. It can be shown by a simple transfinite induction that for solvable groups this agrees with the definition given in the introduction. Furthermore, every solvable group with infinite Hirsch length has a subgroup with arbitrary large Hirsch length. In particular, the existence of a finite-dimensional model $X$ for $\underbar EG$ directly implies that the Hirsch rank of the solvable subgroups of $G$ is bounded by $\dim X$. It remains to prove the other direction.

Let $R$ be a fixed commutative ring with unit and let $G\leq GL_n(R)$ be finitely generated with $N\in\bbN$ an upper bound on the Hirsch rank of the solvable subgroups of $G$. Since $G$ is finitely generated, we can assume that $R$ is also finitely generated and let $\mcn, S$ and $\mcp_1,\ldots,\mcp_k$ be as in \cref{lem:domains}.
Furthermore, let $H$ denote the image of $G$ in $GL_n(S)$ and $p\colon GL_n(R)\to GL_n(S)$ the projection. Let $A$ be a finitely generated abelian subgroup of $H$. Then $p^{-1}(A)$ is solvable. This implies that the rank of the finitely generated abelian subgroups of $H$ is also bounded by $N$.

First let us show that $H$ admits a finite-dimensional model for $\underbar EH$. By \cref{lem:domains} $H$ embeds into $GL_n(S/\mcp_1)\times\ldots\times GL_n(S/\mcp_k)$ and since $H$ is finitely generated, we can assume that all the domains $S/\mcp_i$ are as well. Order them in such a way that $S/\mcp_1,\ldots,S/\mcp_q$ are of positive characteristic and $S/\mcp_{q+1},\ldots,S/\mcp_k$ are of characteristic zero. Then $GL_n(S/\mcp_{q+1})\times\ldots\times GL_n(S/\mcp_k)$ embeds into $GL_{n(k-q)}(\bbC)$. Let $\pi$ denote the projection of $H$ to $GL_n(S/\mcp_1)\times\ldots\times GL_n(S/\mcp_q)$ and let $\pi_i$ denote the projection of $H$ to $GL_n(S/\mcp_i)$ for $i=1,\ldots,q$, then $\pi_i(H)$ admits a finite-dimensional model $E_i$ for $\underbar E\pi_i(H)$ by \cite[Corollary 5]{poschar} and thus $E_1\times\ldots\times E_q$ is a finite-dimensional model for $\underbar E\pi(H)$. By \cref{thm:dim} it remains to show that for every finite subgroup $F\in \pi(H)$ the preimage $\pi^{-1}(F)$ admits a finite-dimensional model with dimension bounded uniformly in $F$.
Let $\rho$ denote the projection from $H$ to $GL_{n(k-q)}(\bbC)$. Then $\rho(H)$ is virtually torsion free by Selberg's Lemma \cite{selberg}. The group $\rho(\ker \pi)$ is isomorphic to $\ker \pi$ and thus $N$ is a bound on the rank of its finitely generated abelian subgroups. Furthermore, $\rho(\ker\pi)$ has finite index in $\rho(\pi^{-1}(F))$ for every finite subgroup $F\leq \pi(H)$. Thus, the rank of the finitely generated abelian subgroups of $\rho(\pi^{-1}(F))$ is also bounded by $N$. By \cite[Proposition 3.1]{KasLie} this implies that the rank of the finitely generated unipotent subgroups of $\rho(\pi^{-1}(F))$ is bounded by $\frac{N(N+1)}{2}$. This implies that $\rho(\pi^{-1}(F))$ has finite virtual cohomological dimension bounded uniformly in $F$, see \cite[Remark after Theorem 3.3]{alperin}. The order of the finite subgroups in $\rho(\pi^{-1}(F))$ is bounded uniformly in $F$ since they are all contained inside the virtually torsion-free group $\rho(H)$. By \cite[Theorem 1.10]{MR1757730} this implies that there exist finite-dimensional models for $\underbar E\rho(\pi^{-1}(F))$ with dimension bounded uniformly in $F$ and since $\rho\colon \pi^{-1}(F)\to \rho(\pi^{-1}(F))$ has finite kernel they are also models for $\underbar E\pi^{-1}(F)$. This completes the proof that $H$ admits a finite-dimensional model for $\underbar EH$.

For every finite subgroup $F\leq H$, its preimage $A$ in $G$ is virtually nilpotent and thus elementary amenable, and the Hirsch rank of $A$ is bounded by $N$. By the inequalities from the beginning of the proof this implies that there is a model for $\underbar EA$ of dimension at most $N+2$. Using again \cref{thm:dim} we conclude that there exists a finite-dimensional model for $\underbar EG$.
\end{proof}
\section*{Appendix}
\setcounter{section}{1}
\setcounter{equation}{0}
\renewcommand\thesection{\Alph{section}}
In this appendix we want to prove the following
\begin{thm}[{\cite[Theorem 3.2.2]{thesis}}]
\label{main}
Let $G$ be a discrete group such that $\{H\backslash G\}_{H\in\mcF in}$ has FDC and let $\mcA$ be a small additive $G$-category. Assume that there is a finite-dimensional $G$-CW-model for the classifying space for proper $G$-actions $\underbar EG$.
Then the assembly map in algebraic $K$-theory $H_*^G(\underbar EG;\bbK_\mcA)\rightarrow K_*(\mcA[G])$ is a split injection.
\end{thm}
The analogous result in $L$-theory \cite[Theorem 3.3.1]{thesis} follows in the same way from the results of \cite{KasFDC}. We will use the notation introduced in \cite{KasFDC}. Note that in the appendix metrics are allowed to take on the value $\infty$. We will need the following equivariant version of FDC.
\begin{defi}
\label{equivfamily}
An \emph{equivariant metric family} is a family $\{(X_\alpha,G_\alpha)\}_{\alpha\in A}$ where $G_\alpha$ is a group and $X_\alpha$ is a metric $G_\alpha$-space.
\end{defi}
\begin{defi}
An equivariant metric family $\mcX=\{(X_\alpha,G_\alpha)\}_{\alpha\in A}$ \emph{decomposes over a class of equivariant metric families $\mfD$} if for every $r>0$ and every $\alpha\in A$ there exists a decomposition $X_\alpha=U_{\alpha}^r\cup V_{\alpha}^r$ into $G_\alpha$-invariant subspaces and $r$-disjoint decompositions
\[U_{\alpha}^r=\bigcup^{r\text{-disj.}}_{i\in I(r,\alpha)}U_{\alpha,i}^{r}\quad\text{and}\quad V_{\alpha}^r=\bigcup^{r\text{-disj.}}_{j\in J(r,\alpha)}V_{\alpha,j}^{r},\]
such that $G_\alpha$ acts on $I(r,\alpha)$ and $J(r,\alpha)$ and for every $g\in G_\alpha$ we have $gU^{r}_{\alpha,i}=U^{r}_{\alpha,gi}$ and $gV^{r}_{\alpha,j}=V^{r}_{\alpha,gj}$. Furthermore, the families \[\Bigg\{\bigg(\coprod_{i\in I(r,\alpha)}U_{\alpha,i}^{r},G_\alpha\bigg)\Bigg\}_{\alpha\in A}\text{ and}\quad\Bigg\{\bigg(\coprod_{j\in J(r,\alpha)}V_{\alpha,j}^{r},G_\alpha\bigg)\Bigg\}_{\alpha\in A}\] have to lie in $\mfD$.

Notice that the underlying sets of $U_\alpha^r$ and $\coprod_{i\in I(r,\alpha)}U_{\alpha,i}^{r}$ are canonically isomorphic and in this sense the $G_\alpha$-action on $\coprod_{i\in I(r,\alpha)}U_{\alpha,i}^{r}$ is the same as the action on $U_\alpha^r$, only the metric has changed.
\end{defi} 
\begin{defi}
An equivariant metric family $\mcX$ is called \emph{semi-bounded}, if there exists $R>0$ such that for all $(X,G)\in\mcX$ and $x,y\in X$ we have $d(x,y)<R$ or $d(x,y)=\infty$.

Let $e\mfB$ denote the class of semi-bounded equivariant families. We set $e\mfD_0=e\mfB$ and given a successor ordinal $\gamma+1$ we define $e\mfD_{\gamma+1}$ to be the class of all equivariant metric families which decompose over $e\mcD_\gamma$. For a limit ordinal $\lambda$ we define
\[e\mfD_\lambda=\bigcup_{\gamma<\lambda}e\mfD_\gamma.\]
An equivariant metric family $\mcX$ has \emph{finite decomposition complexity (FDC)} if $\mcX$ lies in $e\mfD_\gamma$ for some ordinal $\gamma$.

Note that the equivariant metric family $\{(X_\alpha,\{e\})\}_{\alpha\in A}$ has FDC if and only if the metric family $\{X_\alpha\}_{\alpha\in A}$ has FDC.  
\end{defi}
A metric family $\{X_\alpha\}_{\alpha\in A}$ has uniformly bounded geometry if for every $R>0$ there exists $N\in\bbN$ such that for every $\alpha\in A, U\subseteq X_\alpha$ with $\diam(U)\leq R$ the set $U$ contains at most $N$ elements.

The following is a generalization of Ramras, Tessera and Yu \cite[Theorem 6.4]{k-theory}. The proof is analogous to the proof of \cite[Theorem 6.4]{k-theory} and can be found in \cite{thesis}. The additive $G$-category $\mcA_{G}(X)$ is defined in \cite[Definition 5.1]{KasFDC} and $\mcA^G_G(X)$ denotes the fixed-point category. For a definition of the bounded product see \cite[Definition 5.11]{KasFDC}.
\begin{thm}
\label{thm:rty}
Let $\mcX=\{(X_\alpha,G_\alpha)\}_{\alpha\in A}$ be an equivariant family with FDC, and let the family $\{X_\alpha\}_{\alpha\in A}$ have bounded geometry uniformly. Then
\[\colim_sK_n\Bigg(\prod^{bd}_{\alpha\in A}\mcA^{G_\alpha}_{G_\alpha}(P_sX_\alpha)\Bigg)=0\]
for all $n\in\bbZ$, where the colimit is taken over the maps induced by the inclusion of the respective Rips complexes $P_sX_\alpha$.
\end{thm}
Furthermore, recall the following
\begin{thm}[{\cite[Theorem 7.6]{KasFDC}}]
\label{descent}
Let $G$ be a discrete group admitting a finite-dimensional model for $\underbar EG$ and let $X$ be a simplicial $G$-CW complex with a proper $G$-invariant metric. Assume that for every $G$-set $J$ with finite stabilizers
\[\colim_{K}K_n\Big(\prod^{bd}_{j\in J}\mcA_G(GK)\Big)^G=0,\]
where the colimit is taken over all finite subcomplexes $K\subseteq X$.
Then the assembly map
\[H_*^G(X;\bbK_\mcA)\rightarrow K_*(\mcA[G])\]
is a split injection.
\end{thm}
Before we can prove \cref{main} we need the following.
\begin{prop}[{\cite[Proposition 3.2.1]{thesis}}]
\label{eFDC}
Let $G$ be a group such that the metric family $\{H\backslash G\}_{H\in\mcF in}$ has FDC. Then the equivariant metric family $\{(G,H)\}_{H\in\mcF in}$ has FDC as well. 
\end{prop}
\begin{proof}
Let $\{(X_i,G_i)\}_{i\in I}$ be an equivariant metric family with $G_i\leq G$ a finite subgroup and assume $X_i\subseteq \coprod_{A_i}G$ is a $G_i$-invariant subspace, where $A_i$ is a $G_i$-set.
We prove by induction on the decomposition complexity that the family $\{(X_i,G_i)\}_{i\in I}$ lies in $e\mfD_{\gamma+1}$ if $\{G_i\backslash X_i\}_{i\in I} \in\mfD_\gamma$. For the start of the induction let $\{G_i\backslash X_i\}_{i\in I}$ be in $\mfD_0=\mfB$. Since $G_i\backslash X_i$ is bounded,  there is $a_i\in A_i$ with $X_i\subseteq \coprod_{G_ia_i}G$. Then there exist $R>0$ and $Y_i\subseteq G=\coprod_{\{a_i\}}G\subseteq \coprod_{A_i}G$ with $\diam Y_i<R$ for all $i\in I$ such that $X_i=G_iY_i\subseteq \coprod_{A_i}G$. Let $G_i'\subseteq G_i$ be the stabilizer of $a_i$. Then this is equivalent \[X_i\cong\coprod_{[g]\in G_i/G_i'}gG_i'Y_i\]
with $G_i'Y_i\subseteq G$.

Let $r>0$ be given and define $\mcS_r:=\{H\in\mcF in\mid H=\langle S\rangle, S\subseteq B_{2R+r}(e)\}$ and $k:=\max_{H\in\mcS}|H|$. Let $g_i\in Y_i$ be a fixed base point. Let $H_i\leq G_i'$ be the subgroup generated by those $g\in G_i'$ with $d(Y_i,gY_i)<r$. For these $g$ we have $d(e,g_i^{-1}gg_i)<2R+r$. Therefore, $g_i^{-1}H_ig_i\in\mcS_r$ and $|H_i|\leq k$. We have the decomposition
\[X_i=\bigcup_{[g]\in G_i/H_i}gH_iY_i.\]
This decomposition is $r$-disjoint, since $d(ghy,g'h'y')<r$ with $g,g'\in G_i,h,h'\in H_i$ and $y,y'\in Y_i$ implies that $d(Y_i,h^{-1}g^{-1}g'h'Y_i)<r$ and so by definition the element $h^{-1}g^{-1}g'h'$ lies in $H_i$ which is equivalent to $gH_i=g'H_i$.

By definition of $H_i$ each $h\in H_i$ can be written as $h=g_1\cdots g_l$ with $l\leq |H|\leq k$ and $g_j$ such that $d(Y_i,g_jY_i)<r$. For every $y,y'\in Y_i$ by left-invariance and the triangle inequality we obtain
\begin{align*}
d(y,hy')&\leq d(y,g_1y')+d(g_1y',g_1g_2y')+\ldots d(g_1\cdots g_{l-1}y',hy')\\&=d(y,g_1y')+d(y',g_2y')+\ldots d(y',g_ly')< lr.\end{align*}
Therefore $\diam gH_iY_i=\diam H_iY_i<kr$. Thus, $\{(X_i,G_i)\}_{i\in I}$ is $r$-decomposable over $e\mfD_0=e\mfB$ for every $r>0$ and lies in $e\mfD_1$.

If $\{G_i\backslash X_i\}_{i\in I}$ lies in $\mfD_{\gamma+1}$, then it decomposes over $\mfD_\gamma$ and the preimages under the projection $X_i\to G_i\backslash X_i$ give a decomposition of $\{(X_i,G_i)\}$ over $e\mfD_{\gamma+1}$ by the induction hypothesis. Here $G_i$ acts trivially on the index set of the decomposition. The induction step for limit ordinals is trivial.
\end{proof}
\begin{proof}[Proof of \cref{main}]
By \cite[Proposition 1.5]{KasFDC} $G$ admits a finite dimensional model $X$ for $\underbar EG$ with a left-invariant proper metric. By \cref{descent} we have to show that
\[\colim_KK_n\bigg(\prod_{j\in J}^{bd}\mcA_G(GK)\bigg)^G=0,\]
where the colimit is taken over all finite subcomplexes $K\subseteq X$.
Since the category $\big(\prod_{j\in J}^{bd}\mcA_G(GK)\big)^G$ is equivalent to $\prod_{G j\in G\backslash J}^{bd}\mcA_G^{G_j}(P_sG)$, where $G_j$ is the stabilizer of $j\in J$, this is equivalent to showing that for every family of finite subgroups~$\{G_i\}_{i\in I}$ over some index set $I$ the following holds
\[\colim_KK_n\bigg(\prod_{i\in I}^{bd}\mcA^{G_i}_G(GK))\bigg)=0.\]
By \cite[Lemma 1.8]{KasFDC} and \cite[Proposition 6.3]{KasFDC}, for every finite subcomplex $K\subseteq X$ there exists $K'\subseteq X$ finite and $s>0$ and maps $GK\to P_s(G)\to GK'$ such that the composition is metrically homotopic to the identity. In particular, the composition induces the identity in the $K$-theory of the associated controlled categories. Thus it remains to show
\[\colim_sK_n\bigg(\prod_{i\in I}^{bd}\mcA^{G_i}_G(P_sG))\bigg)=0.\]
Since $\{(G,H)\}_{H\in\mcF in}$ has FDC by \cref{eFDC} and the category $\mcA_G^{G_i}(P_sG)$ is equivalent to $\mcA_{G_i}^{G_i}(P_sG)$, this follows from \cref{thm:rty}.
\end{proof}
\bibliographystyle{amsalpha}
\bibliography{InjectivityLinearGroups}

\providecommand{\bysame}{\leavevmode\hbox to3em{\hrulefill}\thinspace}
\providecommand{\MR}{\relax\ifhmode\unskip\space\fi MR }
\providecommand{\MRhref}[2]{%
  \href{http://www.ams.org/mathscinet-getitem?mr=#1}{#2}
}
\providecommand{\href}[2]{#2}
\begin{thebibliography}{GHW05}

\bibitem[Alp87]{alperinselberg}
Roger~C. Alperin, \emph{An elementary account of {S}elberg's lemma}, Enseign.
  Math. (2) \textbf{33} (1987), no.~3-4, 269--273. \MR{925989 (89f:20051)}

\bibitem[AS82]{alperin}
Roger~C. Alperin and Peter~B. Shalen, \emph{Linear groups of finite
  cohomological dimension}, Invent. Math. \textbf{66} (1982), no.~1, 89--98.
  \MR{652648 (84a:20052)}

\bibitem[BR07]{coefficients}
Arthur Bartels and Holger Reich, \emph{Coefficients for the {F}arrell-{J}ones
  conjecture}, Adv. Math. \textbf{209} (2007), no.~1, 337--362. \MR{2294225
  (2008a:19002)}

\bibitem[DL98]{davislueck}
James~F. Davis and Wolfgang L{\"u}ck, \emph{Spaces over a category and assembly
  maps in isomorphism conjectures in {$K$}- and {$L$}-theory}, $K$-Theory
  \textbf{15} (1998), no.~3, 201--252. \MR{1659969 (99m:55004)}

\bibitem[DP]{poschar}
Dieter Degrijse and Nansen Petrosyan, \emph{{Bredon cohomological dimensions
  for groups acting on {CAT(0)}-spaces}}, arXiv:1208.3884, to appear in Groups,
  Geometry, and Dynamics.

\bibitem[FN07]{MR2280168}
Ram{\'o}n~J. Flores and Brita E.~A. Nucinkis, \emph{On {B}redon homology of
  elementary amenable groups}, Proc. Amer. Math. Soc. \textbf{135} (2007),
  no.~1, 5--11 (electronic). \MR{2280168 (2008c:20095)}

\bibitem[GHW05]{GHWNovikovlinear}
Erik Guentner, Nigel Higson, and Shmuel Weinberger, \emph{The {N}ovikov
  conjecture for linear groups}, Publ. Math. Inst. Hautes \'Etudes Sci. (2005),
  no.~101, 243--268. \MR{2217050 (2007c:19007)}

\bibitem[GTY12]{rigidity}
Erik Guentner, Romain Tessera, and Guoliang Yu, \emph{A notion of geometric
  complexity and its application to topological rigidity}, Invent. Math.
  \textbf{189} (2012), no.~2, 315--357. \MR{2947546}

\bibitem[GTY13]{FDC}
\bysame, \emph{Discrete groups with finite decomposition complexity}, Groups
  Geom. Dyn. \textbf{7} (2013), no.~2, 377--402. \MR{3054574}

\bibitem[Kas]{KasLie}
Daniel Kasprowski, \emph{On the {K}-theory of subgroups of virtually connected
  {L}ie groups}, arXiv:1409.3452, to appear in Algebraic and Geometric
  Topology.

\bibitem[Kas14]{thesis}
\bysame, \emph{On the {$K$}-theory of groups with finite decomposition
  complexity}, Ph.D. thesis, Westf\"alische Wilhelms Universtit\"at M\"unster,
  2014.

\bibitem[Kas15]{KasFDC}
\bysame, \emph{On the {K}-theory of groups with finite decomposition
  complexity}, Proceedings of the London Mathematical Society \textbf{110}
  (2015), no.~3, 565--592.

\bibitem[L{\"u}c89]{transformation}
Wolfgang L{\"u}ck, \emph{Transformation groups and algebraic {$K$}-theory},
  Lecture Notes in Mathematics, vol. 1408, Springer-Verlag, Berlin, 1989,
  Mathematica Gottingensis. \MR{1027600 (91g:57036)}

\bibitem[L{\"u}c00]{MR1757730}
\bysame, \emph{The type of the classifying space for a family of subgroups}, J.
  Pure Appl. Algebra \textbf{149} (2000), no.~2, 177--203. \MR{1757730
  (2001i:55018)}

\bibitem[Nuc04]{MR2061566}
Brita E.~A. Nucinkis, \emph{On dimensions in {B}redon homology}, Homology
  Homotopy Appl. \textbf{6} (2004), no.~1, 33--47. \MR{2061566 (2005c:20092)}

\bibitem[PW85]{MR802790}
Erik~K. Pedersen and Charles~A. Weibel, \emph{A nonconnective delooping of
  algebraic {$K$}-theory}, Algebraic and geometric topology ({N}ew {B}runswick,
  {N}.{J}., 1983), Lecture Notes in Math., vol. 1126, Springer, Berlin, 1985,
  pp.~166--181. \MR{802790 (87b:18012)}

\bibitem[Roe03]{roe}
John Roe, \emph{Lectures on coarse geometry}, University Lecture Series,
  vol.~31, American Mathematical Society, Providence, RI, 2003. \MR{2007488
  (2004g:53050)}

\bibitem[RTY14]{k-theory}
Daniel~A. Ramras, Romain Tessera, and Guoliang Yu, \emph{Finite decomposition
  complexity and the integral {N}ovikov conjecture for higher algebraic
  {$K$}-theory}, J. Reine Angew. Math. \textbf{694} (2014), 129--178.
  \MR{3259041}

\bibitem[Sel60]{selberg}
Atle Selberg, \emph{On discontinuous groups in higher-dimensional symmetric
  spaces}, Contributions to function theory (internat. {C}olloq. {F}unction
  {T}heory, {B}ombay, 1960), Tata Institute of Fundamental Research, Bombay,
  1960, pp.~147--164. \MR{0130324 (24 \#A188)}

\bibitem[Weg15]{solvable}
Christian Wegner, \emph{The {F}arrell-{J}ones {C}onjecture for virtually
  solvable groups}, J. Topol. (2015), DOI: 10.1112/jtopol/jtv026.

\end{thebibliography}
\end{document}